\documentclass{article}
\usepackage[utf8]{inputenc}
\usepackage[english]{babel}
\usepackage{amsmath}
\usepackage{amsthm}
\usepackage{amssymb}
\usepackage{authblk}
\usepackage{xcolor}
\usepackage{enumitem}
\newtheorem{conj}{Conjecture}
\newtheorem{theorem}{Theorem}[section]
\newtheorem{prop}{Proposition}[section]
\newtheorem{cor}{Corollary}[prop]
\newtheorem{lemma}{Lemma}[section]
\theoremstyle{remark}
\newtheorem*{remark}{Remark}
\def\F{\mathcal F}
\def\H{\mathcal H}
\theoremstyle{definition}
\newtheorem{defn}[equation]{Definition}

\title{Closures of Union-Closed Families}
\author{Dhruv Bhasin\thanks{Department of Mathematics, Indian Institute of Science Education and Research, Pune\\Email id: bhasin.dhruv@students.iiserpune.ac.in} }

\date{\today}

\begin{document}

\maketitle
\begin{abstract}
    Given a union-closed family $\mathcal{F}$ of subsets of the universe $[n]$, with $\mathcal{F}$ not equal to the power set of $[n]$, a new subset $A$ can be added to it such that the resulting family remains union-closed. We construct a new family $\overline{\mathcal{F}}$ by adding to $\mathcal{F}$ all such $A$'s, and call this the closure of $\mathcal F$. This paper is dedicated to the study of various properties of such closures, including characterizing families whose closures equal the power set of $[n]$, providing a criterion for the existence of closure roots of such families etc.   
\end{abstract}

\section{Introduction}
The union-closed sets conjecture is an easy-to-state, notoriously difficult problem that was first proposed by Péter Frankl(see \cite{k}). Several research articles (see for e. g. \cite{a, d, h, i, j}) and a Polymath Project (see \cite{g}) have been devoted to the unraveling of this long-standing open problem. 

Letting $[n]$ denote the finite set $\{1, 2,\dots, n\}$ which will serve as our universe and $2^{[n]}$ denote the power set of $[n]$, a family of subsets $\mathcal F\subseteq2^{[n]}$ is said to be \emph{union-closed over universe $[n]$} if $[n]\in\mathcal F$, and for all $A,B\in\mathcal F$, we have $A\cup B\in\mathcal F$. We call each subset $A$ of $[n]$ that is contained in the family $\mathcal{F}$ a \emph{member-subset} of $\mathcal{F}$, and we let $|\mathcal{F}|$ denote the total number of member-subsets of $\mathcal{F}$. Then the union-closed sets conjecture can be stated as follows:
\begin{conj}\label{main_conj}
Let $\mathcal F$ be a union-closed family of sets over the universe $[n]$. Then there exists some element $i\in[n]$ such that $i$ belongs to at least half of the member-subsets of $\mathcal F$, i.e.\ $\sum_{A \in \mathcal{F}} \mathbf{1}_{i \in A} \geqslant \lfloor |\mathcal{F}|/2\rfloor$. 
\end{conj}
A detailed survey by Bruhn and Schaudt in \cite{c} explains how the conjecture has travelled, both geographically and mathematically, as they put it. The problem has equivalent formulations in lattice theory (see \cite{i}) and graph theory (see \cite{d}), and the claim in the conjecture has been established for various lattice classes and graph classes. Some results have been obtained proving that when the size $|\mathcal{F}|$ of the family is sufficiently large (as a function of the size of the universe, $n$), then $\mathcal{F}$ satisfies the union-closed sets conjecture (for instance, see \cite{h}, where it has been established, using Boolean analysis, that there exists a constant $c$ such that all union-closed families $\mathcal{F}$ with $|\mathcal F|\geq(\frac{1}{2}-c)2^{n-1}$ satisfy the conjecture). 

In this paper, we investigate several intriguing properties of union-closed families. In \cite{b}, the number of non-isomorphic union-closed families for $n=7$ was computed, and an algorithm was devised which involved recursively adding a new set $A$ to a union-closed family $\mathcal F$ such that the new family $\mathcal F\cup \{A\}$ remains union-closed. We collect all such sets $A$ and construct a new family 
\begin{equation}\label{closure_defn}
\overline{\mathcal F}=\{A\in 2^{[n]}: \mathcal F\cup\{A\} \text{ is union-closed}\}.
\end{equation}
We call this the \emph{closure} of $\mathcal F$. The family $\overline{\mathcal F}$ is itself union-closed. We mention here (see also Lemma~\ref{sec:2-lem_2}) that as long as a family $\mathcal F$ is a proper subset of $2^{[n]}$, there exists an $A\notin\mathcal F$ such that $\mathcal F\cup \{A\}$ remains union-closed. Consequently, for every union-closed $\mathcal{F}$ that is a proper subset of $2^{[n]}$, it is also a proper subset of its closure $\overline{\mathcal{F}}$ defined in \eqref{closure_defn}. 

Note, further, that since the universe $[n]$ itself is finite, one requires to repeat the operation of taking closures only a finite number of times before a given union-closed family $\mathcal{F}$ reaches the power set $2^{[n]}$. In other words, if we set $\overline{\mathcal F}^{(0)}=\mathcal F$ and $\overline{\F}^{(i)}=\overline{\overline{\mathcal F}^{(i-1)}}$ for each $i \in \mathbb{N}$, then there is a finite $k \in \mathbb{N}$ such that $\overline{\mathcal{F}}^{(k)} = 2^{[n]}$. This inspires us to introduce the following definition:
\begin{defn}
A union-closed family $\mathcal F$ is said to be \emph{$k$-dense}, for $k \in \mathbb{N}_{0}$, if $k$ is the smallest non-negative integer such that $\overline{\mathcal{F}}^{(k)} = 2^{[n]}$. We call $k$ the \emph{density} of $\mathcal{F}$. 
\end{defn}
Closures provide a natural parameter for a potential induction-based proof of Conjecture~\ref{main_conj}, namely the density of a family. The base case, $0-$dense families, is trivially true as $2^{[n]}$ is the only $0-$dense family. Thus, Conjecture~\ref{main_conj} is equivalent to the following conjecture:
\begin{conj}\label{main_conj_2}
Suppose $\F$ is a union-closed family such that $\overline{\F}$ satisfies Conjecture~\ref{main_conj}. Then $\F$ also satisfies Conjecture~\ref{main_conj}.
\end{conj}
The benefit of this formulation is that it gives us some extra information about the union-closed family for which we need to check the validity of Conjecture~\ref{main_conj}. If the structure of closures and its properties can be understood fairly well then it can hold a potential to give a better understanding of Conjecture~\ref{main_conj}.

We now describe the organization of this paper. We mention at the very outset that throughout this work, it has been assumed that the empty set $\emptyset$ is not contained in \emph{any} family $\mathcal{F}$ of subsets of $[n]$, including the cases where $\mathcal{F} = 2^{[n]}$. In \S\ref{sec:2}, we state and prove several properties of closures and densities of union-closed families of $[n]$, many of which are further utilized in the proofs of results in \S\ref{sec:k}, \S\ref{sec:example}, \S\ref{sec:3}. We conclude \S\ref{sec:2} by Theorem~\ref{main_1} which gives a criterion to obtain a lower bound on the density of a given union-closed family $\F$. We provide examples to show that for some families this bound can be achieved while for others it need not always be achieved. 

In \S\ref{sec:k} we show that for each $k$ there are at least $\binom{n}{k-1}f_{k-1}$ non-isomorphic union-closed families that are $k-$dense, where $n$ is the size of the universe, $k\leqslant n-1$ and $f_k$ is the number of labelled union-closed families over universe $k$. We achieve this by explicit construction. In \S\ref{sec:example}, for each $n\geqslant 6$, we give example of an $(n-1)-$dense family $\F$ which is different from those discussed in \S$\ref{sec:k}$. We give a complete description of its closures $\overline{\F}^{(k)}$ for $k\leq n-5$ and consequently show that it is $(n-1)-$dense. The method of proof used in \S\ref{sec:k} and $\S\ref{sec:example}$ can potentially be used to find densities and description of closures of various union-closed families.

The goal of \S\ref{sec:3} is to give a criterion to check whether a given $1-$dense family $\F$ has a \textit{closure root} or not, i.e, $\H\subseteq\F$ such that $\overline{\H}=\F$. We achieve this using the notion of \textit{relative subsets} we introduce in $\S\ref{sec:3}$. Intuitively, relative subsets can be thought of as subsets within a given family. Using this notion we also generalise Lemma~\ref{sec:2-lem_4} to Proposition~\ref{sec:3_prop_1}. Having build up some basic results regarding relative subsets, in Theorem~\ref{sec:3_theorem_2} , for a $1-$dense family $\F,$ we construct a particular family such that $\F$ has a closure root $\Leftrightarrow$ that particular family is a closure root of $\F$. We conclude the section by Corollary~\ref{sec:3_example_lemma} in which we give a non-trivial example of a $1-$dense family having no closure roots. The notion of relative subsets can serve as an effective tool in the study of closures of union-closed families.

\section{Various properties of closures and densities of union-closed families}\label{sec:2}
 We begin by stating several properties of closures and densities of union-closed families, followed by a succinct characterization of $1$-dense union-closed families using \emph{up-sets}, and finally end with Theorem~\ref{main_1}.
\begin{lemma}\label{sec:2-lem_1}
\label{1}
Let $\mathcal F$ be a union-closed family and $\overline{\mathcal F}$ be its closure, as defined in \eqref{closure_defn}. Then $\overline{\mathcal F}$ is union-closed.
\end{lemma}
\begin{proof}
Fix any $A, B\in\overline{\mathcal F}$. We wish to show that $A\cup B\in\overline{\mathcal F}$. Fix any $C\in\mathcal F$. As $B \in \overline{\mathcal{F}}$, hence $B\cup C\in\mathcal F\cup\{B\}$, which implies that either $B\cup C=B$ or $B\cup C\in\mathcal F$. In the former case, $A\cup B\cup C=A\cup B\in\mathcal F\cup\{A\cup B\}$. In the latter case, $B\cup C\in\mathcal F$, which in turn implies, since $A \in \overline{\mathcal{F}}$, that $A\cup (B\cup C)\in\mathcal F\cup\{A\}$. If $A\cup B\cup C\in\mathcal F$ then we get our desired conclusion. Otherwise, we have $A\cup B\cup C=A$, which implies $B\subseteq A$ and hence $A\cup B=A$, which we already know is in $\overline{\mathcal F}$.
\end{proof}

\begin{lemma}\label{sec:2-lem_2}
\label{2}
Let $\mathcal F$ be a union-closed family over a universe $[n]$. Suppose that $\mathcal F\neq 2^{[n]}$ then $\mathcal F\subsetneq\overline{\mathcal F}$.  
\end{lemma}
\begin{proof}
Choose $A$ to be a \emph{maximal} subset in the family $2^{[n]} \setminus \mathcal{F}$, i.e.\ if $B \in 2^{[n]} \setminus \mathcal{F}$ and $A \subseteq B$, then $A = B$. We claim that $\mathcal{F} \cup \{A\}$ is union-closed. To see this, choose any $C \in \mathcal{F}$. If $C \subset A$, then $C \cup A = A$, and if $C \not\subset A$, then $A$ is a proper subset of $A \cup C$, which in turn, via the maximality of $A$, implies that $A \cup C \in \mathcal{F}$. Either way, we end up with $A \cup C \in \mathcal{F} \cup \{A\}$, as desired.
\end{proof}


At this point, we define $\mathcal A_{k,n}=\{A\subseteq {[n]}:|A|=k\}$ to be the collection of all subsets of $[n]$ which have cardinality $k$. 
\begin{lemma}\label{sec:2-lem_3}
\label{3}
Let $\mathcal F$ be a union-closed family over the universe $[n]$, with density $k$. Then $0\leqslant k\leqslant n-1$.
\end{lemma}
\begin{proof}
Recall the definition of $\overline{\mathcal{F}}^{(i)}$ for any $i \in \mathbb{N}_{0}$. From the definition of universe, we know that $\left\{[n]\right\} = \mathcal A_{n,n}\subseteq \overline{\mathcal{F}}^{(i)}$ for each $i$. Let $t_{i}$ denote the smallest positive integer such that $\mathcal A_{t_{i},n}\subseteq \overline{\mathcal{F}}^{(i)}$. Consequently if a subset $A$ belongs to $\overline{\mathcal{F}}^{(i)} \setminus \mathcal A_{t_{i}-1,n}$, then $A$ has to be a maximal element of $2^{[n]} \setminus \overline{\mathcal{F}}^{(i)}$ and hence, as illustrated in the proof of Lemma~\ref{sec:2-lem_2}, $A\in\overline{\mathcal{F}}^{(i+1)}$. Thus we conclude that $\mathcal A_{t_{i}-1,n}\subseteq\overline{\mathcal{F}}^{(i+1)}$. This in turn shows that $t_{i+1} \leqslant t_{i}-1$ for each $i$ with $t_i\geqslant 2$, and therefore, we must have $t_{n-1} = 1$. This implies that $\mathcal{A}_{t,n} \subset \overline{\mathcal{F}}^{(n-1)}$ for all $t \geqslant 1$, so that $\overline{\mathcal{F}}^{(n-1)} = 2^{[n]}$, thus concluding the proof. 
\end{proof}

Therefore, given any family $\mathcal F$ over the universe $[n]$, we obtain $2^{[n]}$ in at most $n-1$ ``steps" of closures. Now, we have:

\begin{cor}\label{sec:2:cor_1}
Let $\F$ be a union-closed family over the universe $[n]$. Then $\mathcal A_{n-k}\subseteq \overline{\F}^{(k)}$.
\end{cor}

This follows from the inequality obtained in the proof of the previous lemma, namely $t_{i+1}\leqslant t_i-1$ alongwith the observation that $t_1\leqslant n-1$. We next show that you cannot do better than this $n-1$, i.e.\ there exists at least one family $\mathcal{F}$ with density precisely $n-1$. 

\begin{prop}\label{sec:2-prop_1}
Consider the family $\mathcal F=\{[1], [2], [3],\dots, [n]\}$ over the universe $[n]$. Then $\mathcal F$ is $(n-1)$-dense.
\end{prop}

\begin{proof}
\sloppy First of all, note that $\mathcal{F}$ is union-closed. Next, note that $\{1, 2, 3,\dots, n-2, n\}\notin\mathcal F$ and $\{1, 2, 3,\dots, n-2\}\in\mathcal F$. These together imply that 
\begin{equation}\label{obs_1}
\{1, 2, 3,\dots, n-3, n\}\notin \overline{\mathcal F}.
\end{equation}
Next, we note that $\{1, 2, 3,\dots, n-3\}\in \mathcal{F} \subset \overline{\mathcal F}$, and this, combined with \eqref{obs_1}, yields
\begin{equation}\label{obs_2}
\{1,2,3,\dots,n-4,n\}\notin \overline{\mathcal{F}}^{(2)}.
\end{equation}
Proceeding thus, by induction, we see that $\{n\}\notin \overline{\mathcal{F}}^{(n-2)}$, which implies that $\overline{\mathcal{F}}^{(n-2)} \neq 2^{[n]}$, and hence the density of $\mathcal{F}$ is at least $n-1$. This, along with Lemma~\ref{sec:2-lem_3}, concludes the proof.
\end{proof}

Let $\mathcal U_n$ denote the set of all union-closed families over universe $[n]$. If we construct a graph $G_n$ whose vertex set is $\mathcal U_n$, and draw a directed edge between two families $\mathcal F_1$ and $\mathcal F_2$, directed from $\mathcal F_1$ to $\mathcal F_2$, iff $\overline{\mathcal F_1}=\mathcal F_2$, then, by Lemma~\ref{sec:2-lem_1}, every vertex has exactly one out-degree. By Lemma~\ref{sec:2-lem_2}, there are no self-loops other than at the vertex $2^{[n]}$. By Lemma~\ref{sec:2-lem_3}, the maximum length of a directed path in $\mathcal G_n$ is $n-1$. So, the set of all union-closed families over universe $[n]$ can be imagined as a tree with root $2^{[n]}$ and depth $n-1$.

Our next result focuses on a nice characterization of $1$-dense families. To this end, we require the well known notion of up-sets(see \cite{l}). 
\begin{defn}
 A family $\mathcal F$ over a given finite or infinite universe is said to be an up-set if for every $A\in\mathcal F$ and $B\supseteq A$, we have $B\in\mathcal F$.
\end{defn}

\begin{lemma}\label{sec:2-lem_4}
Let $\mathcal F\neq 2^{[n]}$ be a union-closed family over the universe $[n]$. Then $\mathcal F$ is $1$-dense if and only if $\mathcal F$ is an up-set. 
\end{lemma}
\begin{proof}
Suppose $\mathcal F$ is $1$-dense. Let $A\in\mathcal F$ and let $A$ be a proper subset of $B$. Since, $\mathcal F$ is $1$-dense, we have $\overline{\mathcal F}=2^{[n]}$, so that $B\setminus A\in\overline{\mathcal{F}}$. Therefore, $\mathcal{F} \cup \{B \setminus A\}$ is union-closed, and hence $B = (B\setminus A)\cup A\in\mathcal F\cup\{B\setminus A\}$. Since $A$ is a non-empty, proper subset of $B$, hence $B \neq B \setminus A$, which yields $B \in\mathcal F$. Hence $\mathcal F$ is an up-set. 

Conversely, suppose that $\mathcal F$ is an up-set. Let $B\in 2^{[n]}-\mathcal F$. Take any $A\in\mathcal F$. Note that $A\cup B\supseteq A$, so that $A\cup B\in\mathcal F$ as $\mathcal{F}$ is an up-set. Therefore, $B\in\overline{\mathcal F}$, thus proving that $\overline{\mathcal F}=2^{[n]}$.
\end{proof}

It is well-known (see for e.g. Introduction of \cite{a}) that up-sets satisfy Conjecture~\ref{main_conj}. By Lemma~\ref{sec:2-lem_4}, we conclude that $1$-dense families satisfy it as well. Using Lemma~\ref{sec:2-lem_4}, we establish Theorem~\ref{main_1}.
\begin{theorem}\label{main_1}
Let $\mathcal F$ be a $k$-dense union-closed family. If there exist subsets $A_1\subsetneq A_2\subsetneq\dots\subsetneq A_r\subsetneq B_r$ of $[n]$ with $A_1, A_2,\dots,A_r\in\mathcal F$ and $B_r\notin\mathcal F$, then $r<k$.
\end{theorem}
\begin{proof} We use induction on the density of $\mathcal F$. First, consider $\mathcal  F$ that is $1$-dense. By Lemma~\ref{sec:2-lem_4}, it is an up-set. Hence $A\in\mathcal F$ and $B\notin\mathcal F$ imply that $A\nsubseteq B$. Therefore, $r=0$. This concludes the proof of the claim for the base case. 

Assume that the result holds for $t$-dense families for some $t \in \mathbb{N}$ with $t \leqslant n-2$. Let $\mathcal F$ be a $t+1$-dense family. Suppose there exist $A_{1}, \ldots, A_{r}$ and $B_{r}$ as in the statement of Theorem~\ref{main_1}. Now, $A_r\in\mathcal F$ and $A_{r-1}\cup(B_r \setminus A_r)\cup A_r=B_r\notin\mathcal F$ together imply that $A_{r-1}\cup(B_r \setminus A_r)\notin\overline{\mathcal F}$. Therefore we get $A_1\subsetneq A_2\subsetneq\dots\subsetneq A_{r-1}\subsetneq A_{r-1}\cup(B_r \setminus A_r)$ with $A_1, A_2,\dots,A_{r-1}\in\overline{\mathcal F}$ and $A_r\cup(B_r-A_r)\notin\overline{\mathcal F}$. Moreover, since $\mathcal F$ is $t+1$-dense, we know that $\overline{\mathcal F}$ is $t$-dense. By our induction hypothesis, we have $r-1<t$ which yields $r<t+1$, as desired.  
\end{proof}

\begin{remark}
Letting $\mathcal F=\{[1], [2], [3],\dots, [n]\}$, $A_i=\{[i]\}$ for $1 \leqslant i \leqslant n-2$ and $B_{n-2}=\{1, 2,\dots, n-2, n\}$, Theorem \ref{main_1} yields the density of $\mathcal F$ to be at least $n-1$, which corroborates the conclusion of Proposition~\ref{sec:2-prop_1}. This also leads to the following, far more general observation: if $\mathcal F$ is a union-closed family such that $[i]\in\mathcal F$ for all $i\in[n]$ and $\{1, 2,\dots, n-2, n\}\notin\mathcal F$, then $\mathcal F$ is $(n-1)$-dense. 
\end{remark}

\sloppy Theorem~\ref{main_1} shows that there cannot exist a ``very long" chain of subsets of the type described in the statement of the theorem, especially when the density of $\mathcal F$ is ``small" compared to $n$, the size of the universe. Let 
$$s(\F)=\max\{r:\exists A_1\subsetneq A_2\subsetneq\dots\subsetneq A_r\subsetneq B_r, A_i\in\F, B_r\notin\F\}. $$
For a $k-dense$ family $\F$, Theorem~\ref{main_1} gives us the lower bound $s(\F)+1\leqslant k$. We note here that this bound is not always tight. As an instance, consider the union-closed family
$$\mathcal F=\{[n-2], \{1, 2, 4, 5, 6,\dots, n\}, \{1, 3, 4, 5,\dots, n\}, \{2, 3,\dots, n\}, [n]\}.$$
It is straightforward to check that $2^{[n-2]}\subseteq\overline{\mathcal F}$ and $\{1, 2, 3,\dots, n-3, n\}\notin\overline{\mathcal F}$. Setting $A_i=[i]$ for $1 \leqslant i \leqslant n-4$ and $B_{n-3}=\{1, 2, 3,\dots, n-3, n\}$, we conclude, by Theorem~\ref{main_1}, that $\overline{\mathcal F}$ is at least $(n-2)$-dense. This, along with Lemma~\ref{sec:2-lem_3}, yields that $\mathcal F$ is $(n-1)$-dense. However, the only choice of subsets $A$ and $B$ such that $A \in \mathcal{F}$, $B \notin \mathcal{F}$ and $A \subsetneq B$ is $A = [n-2]$ and $B = [n-1]$, showing that $s(\F)=1$, which is a lot smaller than $n-1$.

\section{Many $k-$dense families}\label{sec:k}

Let $k$ be a fixed positive integer and $\F$ be a union -closed family over universe $[k]$. In this section, we will consider the family $\H=\F\cup\{[n]\}$ for $n\geqslant k+2$. Throughout this section, whenever $A\in 2^{[n]}$, by $A_1,$ we will mean $A\cap [k]$ and by $A_2$, we will mean $A\cap([n]\setminus[k])$.  Our goal is to prove that $\H$ is $(k+1)-$dense. We begin with:

\begin{lemma}\label{mlemma_1}
For any $A\in 2^{[n]}$ and for any $1\leqslant t\leqslant k$, at least one of the following is true:
\begin{enumerate}[label=\textup{(\roman*)}, noitemsep]
    \item\label{mi} $A_2=\emptyset$, $A_1\in\overline{\F}^{(t)}$.
    \item\label{mii} $A_2=\emptyset$, $A_1\notin\overline{\F}^{(t)}$.
    \item\label{miii} $A_2=[n]\setminus[k]$ and $\forall E\in\overline{\F}^{(t-1)}$ with $E\nsubseteq A$, $E\cup A\in\overline{\H}^{(t-1)}$.
    \item\label{miv} $A_2=[n]\setminus [k]$ and $\exists E\in\overline{\F}^{(t-1)}$ with $E\nsubseteq A$, $E\cup A\notin\overline{\H}^{(t-1)}$.
    \item\label{mv} $|A_1|\geqslant k-t+1$.
    \item\label{mvi} $A_2\neq\emptyset, A_2\neq[n]\setminus[k]$ and $|A_1|\leqslant k-t$.
\end{enumerate}
where we put $\overline{\F}^{0}=\F$ and $\overline{\H}^{(0)}=\H$.
\end{lemma}
\begin{proof}
When $A_2=\emptyset$ then $A$ satisfies one of $1$ or $2$. When $A_2=[n]\setminus[k]$ then $A$ satisfies one of $3$ or $4$. When $A_2\neq\emptyset, A_2\neq[n]\setminus[k]$ then satisfies $A$ satisfies $5$ or $6$.
\end{proof}
Using the previous Lemma, we prove the following:
\begin{theorem}\label{mtheorem_1}
Given $A\in 2^{[n]}$ and a positive integer $t\leqslant k$, if $A$ satisfies any of the conditions stated in \ref{mi}, \ref{miii}, \ref{mv} of Lemma~\ref{mlemma_1} then $A\in\overline{\H}^{(t)}$. On the other hand,if $A$ satisfies any of the conditions stated in \ref{mii}, \ref{miv}, \ref{mvi} of Lemma~\ref{mlemma_1} then $A\notin\overline{\H}^{(t)}$.
\end{theorem}
\begin{proof}
The proof is via induction on $t$. Let $\mathcal C_t=\{A\in\overline{\H}^{(t)}:A_2=[n]\setminus[k]\}$. And put $M=\{\ref{mi}, \ref{miii}, \ref{mv}\}$ and $N=\{\ref{mii}, \ref{miv}, \ref{mvi}\}$. For $q\in M$ let $P(l,q)$ be the statement that: if $A$ satisfies $q$ of Lemma~\ref{mlemma_1} for $t=l$ then $A\in\overline{\H}^{(l)}$. And for $q\in N$ let $P(l,q)$ be the statement that: if $A$ satisfies $q$ of Lemma~\ref{mlemma_1} for $t=l$ then $A\notin\overline{\H}^{(l)}$. Finally, let $P(l,\text{(vii)})$ be the statement that: $\mathcal C_l$ is an up-set. To prove this theorem, we will prove the following statement:
$$P(t)=\bigwedge_{q\in M\cup N\cup\{\text{(vii)}\}}P(t,q)$$
for $1\leq t\leq k$. Let us first establish the base case, i.e., $P(1)$:
\begin{enumerate}
    \item We need to show that $\H\cup\{A\}$ is union-closed. But this is true since $A\in\overline{\F}^{(1)}$ and $\H=\F\cup\{[n]\}$.
    \item Since $A\notin\overline{\F}$, there is a $B\in\F$ such that $A\cup B\notin\F\cup\{A\}$. Since, $A\cup B\subseteq[k]$, we obtain that $A\cup B\notin\H\cup\{A\}\Rightarrow A\notin\overline{\H}$ as required. 
    \item We need to show that $\H\cup\{A\}$ is union-closed. Since $B\in\H$ means that either $B\in\F$ or $B=[n],$ we have that $A\cup B\in \H\cup\{A\}$ in both the cases.
    \item When $A$ satisfies \ref{iv} of Lemma~\ref{mlemma_1} then there is an $E\in\F$ with $E\nsubseteq A$ such that $E\cup A\notin\H$. But since $E\in\F\Rightarrow E\in\H,$ we get that $A\notin\overline{\H}$.
    \item Since $t=1,$ the only choice for $A_1$ is $[k]$. It is easy to see that $A\in\overline{\H}$.
    \item Since $[k]\in\H$ and $A_2\neq\emptyset$ and $A_2\neq [n]\setminus[k]$, we have that $[k]\cup A\notin\H$ and hence, $A\notin\overline{\H}$.
    \item Note that by $P(1,\ref{miii})$ and $P(1,\ref{miv})$, we have that when $A_2=[n]\setminus[k]$ then $A\in\overline{\H}\Leftrightarrow\forall E\in\F$ with $E\nsubseteq A$, $E\cup A\in\H$. So suppose we have $A\in\mathcal C_1$ with $A\subseteq B$. Given $E\in\F$ with $E\nsubseteq B$, we also obtain that $E\nsubseteq A$. Therefore, $E\cup A\in\H$. Since, $A_2=[n]\setminus[k]$, we have that $E\cup A=[n]\Rightarrow E\cup B=[n]\in\H\Rightarrow B\in\mathcal C_1$. And hence, $\mathcal C_1$ is an up-set.
\end{enumerate}
Now, we suppose that $P(r)$ is true for some $1\leqslant r<k$. Let us prove $P(r+1)$:
\begin{enumerate}
\item We are given that $A\in\overline{\F}^{(r+1)}$. We have to investigate $A\cup B$ for $B\in\overline{\H}^{(r)}$. By Lemma~\ref{mlemma_1} and the induction hypothesis, the only cases are: when $B$ satisfies \ref{mi}, \ref{miii}, \ref{mv} of Lemma~\ref{mlemma_1}. We now consider all these cases:
\begin{enumerate}[noitemsep]
    \item Suppose $B$ satisfies \ref{mi} of Lemma~\ref{mlemma_1}. Thus, by the definition of closure, $A\cup B\in\overline{\F}^{(r)}\cup\{A\}$. If $A\cup B=A$ then there is nothing further to investigate. On the other hand, if $A\cup B\in\overline{\F}^{(r)}$ then by $P(r,\ref{mi})$, we have that $A\cup B\in\overline{\H}^{(r)}$, as required.
    \item Suppose $B$ satisfies \ref{miii} of Lemma~\ref{mlemma_1}. Thus by $P(r,\ref{miii})$ we have that $B\in\mathcal C_r$. Therefore by $P(r,\text{(vii)})$, $\mathcal C_r$ is an up-set and hence $A\cup B\in\mathcal C_r\subseteq\overline{\H}^{(r)}$ as required.
    \item Suppose $B$ satisfies \ref{mv} of Lemma~\ref{mlemma_1}. Therefore, $|B_1|\geqslant k-r+1$. And hence, $|(A\cup B)_1|\geqslant k-r+1\Rightarrow A\cup B\in\overline{\H}^{(r)}$.
\end{enumerate}
\item Since $A\notin\overline{\F}^{(r+1)}$, there is a $B\in\overline{\F}^{(r)}$ such that $A\cup B\notin\overline{\F}^{(r)}\cup\{A\}$. By $P(r,\ref{mi})$, we obtain that $B\in\overline{\H}^{(r)}$ and by $P(r,\ref{mi})$, we have $A\cup B\notin\overline{\H}^{(r)}$. Therefore, $A\notin\overline{\H}^{(r+1)}$ as required.  
\item Suppose $A$ satisfies \ref{miii} of Lemma~\ref{mlemma_1}.  We have to investigate $A\cup B$ for $B\in\overline{\H}^{(r)}$. By Lemma~\ref{mlemma_1} and the induction hypothesis, the only cases are when $B$ satisfies \ref{mi}, \ref{miii}, \ref{mv} of Lemma~\ref{mlemma_1}. We now consider all these cases:
\begin{enumerate}[noitemsep]
\item When $B$ satisfies \ref{mi} of the Lemma~\ref{mlemma_1} then $B\in\overline{\F}^{(r)}$ and thus, $B\cup A\in\overline{\H}^{(r)}\cup\{A\}$, as required.
\item Suppose $B$ satisfies \ref{miii} of the Lemma~\ref{mlemma_1}. Thus by $P(r,\ref{miii})$ we have that $B\in\mathcal C_r$. Therefore by $P(r,\text{(vii)})$, $\mathcal C_r$ is an up-set and hence $A\cup B\in\mathcal C_r\subseteq\overline{\H}^{(r)}$ as required.
\item Suppose $B$ satisfies \ref{mv} of Lemma~\ref{mlemma_1}. Therefore, $|B_1|\geqslant k-r+1$. And hence, $|(A\cup B)_1|\geqslant k-r+1\Rightarrow A\cup B\in\overline{\H}^{(r)}$.
\end{enumerate}
\item We are given that $A$ satisfies \ref{miv} of Lemma~\ref{mlemma_1}. Now, since $E\in\overline{\F}^{(r)}$, by $P(r,\ref{i})$, we have that $E\in\overline{\H}^{(r)}$. Therefore, $A\notin\overline{\H}^{(r+1)}$ as required.
\item We are given that $|A_1|\geqslant k-(r+1)+1=k-r$. If $|A_1|\geqslant k-r+1$ then by the induction hypothesis, $A\in\overline{\H}^{(r)}$ and hence there would be nothing further to investigate. Therefore, we only need to look at the case where $|A|=k-r$. If $B\in\overline{\H}^{(r)}$ is such that $B_1\nsubseteq A_1$ then $|(A\cup B)_1|\geqslant k-r+1$ and hence $A\cup B\in\overline{\H}^{(r)}$ by $P(r,\ref{mv})$. On the other hand, if $B_1\subseteq A_1$ then by $P(r,\ref{mvi})$, either $B_2=\emptyset$ or $B_2=[n]\setminus[k]$. If $B_2=\emptyset$ then $A\cup B=B$ and finally if $B_2=[n]\setminus[k]$ then $B\in\mathcal C_r$ and since $\mathcal C_r$ is an up-set by $P(r,\text{(vii)})$, we get that $A\cup B\in\mathcal C_r\subseteq\overline{\H}^{(r)}$, as required.
\item Suppose $A$ satisfies \ref{mvi} of Lemma~\ref{mlemma_1}. Using Corollary~\ref{sec:2:cor_1} we have, $\mathcal A_{k-r,k}\subseteq\overline{\F}^{(r)}$. Let us consider a set $B\in\mathcal A_{k-r,k}$ with $A_1\subseteq B$. By $P(r, \ref{mi})$, this means that $B\in\overline{\H}^{(r)}$. And note that $|(A\cup B)_1|=k-r$ with $(A\cup B)_2=A_2$. Thus, by $P(r,\ref{mvi})$, $A\cup B\notin\overline{\H}^{(r)}$ and clearly $A\cup B\neq A$ since $|A_1|\leqslant k-r-1$. Therefore, $A\notin\overline{\H}^{(r+1)},$ as required.
\item Suppose $A\in\mathcal C_{r+1}$ and $A\subseteq B$. By $P(r+1,\ref{miii})$ and $P(r+1,\ref{miv})$, we have that when $A_2=[n]\setminus[k]$ then
$$ A\in\mathcal C_{r+1}\Leftrightarrow \forall E\in\overline{\F}^{(t)} \text{ with } E\nsubseteq A_1, E\cup A\in\overline{\H}^{(t)}.$$
We need to show that $B\in\mathcal C_{r+1}$. Suppose we have $E\in\overline{\F}^{(t)}$ with $E\nsubseteq B_1$. Since, $A_1\subseteq B_1,$ we also have that $E\nsubseteq A_1$. Therefore, $E\cup A\in\overline{\H}^{(t)}$. This means, by definition, that $E\cup A\in\mathcal C_r$. Using $P(r,\text{(vii)})$, this gives us that $E\cup B\in\mathcal C_r\subseteq\overline{\H}^{(r)}$, thus proving that $\mathcal C_{r+1}$ is an up-set.  
\end{enumerate}

\end{proof}
Finally, let us show that: 
\begin{cor}\label{mcor}
$\H$ is $(k+1)-$dense
\end{cor}
\begin{proof}
Let us consider $2^{[n]}\setminus\overline{\H}^{(k)}$. Suppose $A\in 2^{[n]}\setminus\overline{\H}^{(k)}$. Then by $P(k,\ref{mi})$ and $P(k,\ref{mii})$, $A\notin 2^{[k]}$. By $P(k,\ref{mv})$, $A_1=\emptyset$. By $P(k,\ref{mvi})$, whenever $A_2\neq\emptyset$ and $A_2\neq[n]\setminus[k]$, we have $A\in 2^{[n]}\setminus\overline{\H}^{(k)}$. Thus, the only possibilities for $2^{[n]}\setminus\overline{\H}^{(k)}$ are either $2^{[n]\setminus[k]}$ or $2^{[n]\setminus[k]}\setminus\{[n]\setminus[k]\}$. In either case, $2^{[n]}\setminus\overline{\H}^{(k)}$ is a down set and hence $\overline{\H}^{(k)}$ is an up-set. And since $\overline{\H}^{(k)}\neq 2^{[n]}$, we have that by Lemma~\ref{sec:2-lem_4} $\overline{\H}^{(k)}$ is $1-$dense and hence $\H$ is $(k+1)-$dense. 
\end{proof}
\begin{remark}
If $f_{k}$ is the number of labelled union-closed families with universe $k$ the Corollary~\ref{mcor} tells us that are at least $\binom{n}{k-1}f_{k-1}$ labelled union-closed families that are $k-$dense(since there are $\binom{n}{k-1}$ possible choice for the universe of $\F$). In particular, the number of $(n-1)-$dense union-closed families is at least $\binom{n}{n-2}f_{n-2}$.
\end{remark}
\begin{remark}
Suppose we are given integers $c, k$ and $n$ with $c<k\leqslant n+2$. Then consider the family $\F=\{[k-c], [k-c+1],\dots, [k-1], [n]\}$. It is easy to see that $s(\F)=c$ and it follows from Corollary~\ref{mcor} that $\F$ is $k-$dense. Thus, for all possible feasible combinations of $c, k$ and $n$, there exists a family $\F$ over universe $[n]$ having density $k$ and $s(\F)=c$.
\end{remark}
\section{Another example of an $(n-1)-$dense family}\label{sec:example}

Let $\mathcal F$ be the union-closed family generated by the collection of subsets $\mathcal B=\{\{1, 2\}, \{2, 3\}, \cdots, \{n-1, n\}\}$, i.e.\ 
$$\mathcal F=\left\{\bigcup_{A\in\mathcal C}A: \mathcal{C} \subseteq \mathcal{B}, \mathcal{C} \text{ non-empty}\right\}.$$
Another way of describing $\F$ is as follows: 
$$\F=\{A\subseteq[n]: A\neq\emptyset \text{ and } i\in A \text{ implies that either } i-1 \in A \text{ or }i+1\in A\}.$$
Assume $n\geqslant 6$. This section is devoted to providing a description of the closures of $\mathcal{F}$ and consequently establishing that it is $(n-1)-$dense. We begin our discourse with the following lemma.

\begin{lemma}\label{sec:4-lem_1}
For any $A\notin\F$ and $|A|\leqslant n-k-1$ for some $1\leqslant k\leqslant n-1$, at least one of the following is true:
\begin{enumerate}[label=\textup{(\roman*)}, noitemsep]
    \item \label{i} $\{1\} \subseteq A\subseteq\{1, 3,\cdots , n\}$;
    \item \label{ii} $\{n\} \subseteq A\subseteq\{1, 2, 3,\cdots , n-2, n\}$; 
    \item \label{iii} $\{1, 2\} \subseteq A\subseteq\{1, 2,\cdots , n-1\}$ and $|A|=n-k-1$; 
    \item \label{iv} $\{n-1, n\} \subseteq A\subseteq\{2, 3,\cdots , n\}$ and $|A|=n-k-1$; 
    \item \label{v} $A\subseteq\{2, 3,\cdots , n-1\}$ and $|A|=n-k-1$; 
    \item \label{vi} $\{1, 2, n-1, n\} \subseteq A$ and $|A|=n-k-1$;
    \item \label{vii} $A\subseteq\{1, 2, \dots , n-2\}$, $A\cap\{1, 2\}\neq\emptyset$, $|A|=n-k-2$ and $A\setminus\{\max A\}\notin\F$, where $\max A$ indicates the largest element contained in $A$;
    \item \label{viii} $A\subseteq\{3, 4, \dots , n\}$, $A\cap\{n-1, n\}\neq\emptyset$, $|A|=n-k-2$ and $A\setminus\{\min A\}\notin\F$, where $\min A$ indicates the smallest element contained in $A$; 
    \item \label{ix} $|A|=n-k-2$, $A\cap\{1, 2\}\neq \emptyset$ and $A\cap\{n-1, n\}\neq\emptyset$; 
    \item \label{x} $A\subseteq\{1, 2,\dots , n-2\}$, $|A|=n-k-2$ and $A\setminus\{\max A\}\in\F$;  
    \item \label{xi} $A\subseteq\{3, 4,\dots , n\}$, $|A|=n-k-2$ and $A\setminus\{\min A\}\in\F$;
    \item \label{xii} $A\subseteq\{3, 4, \dots, n-2\}$, $|A|=n-k-2$ and $\exists j\in A$ such that $\min A< j < \max A$ and $j-1, j+1 \notin A$;
    \item \label{xiii} $A\subseteq\{3, 4, \dots, n-2\}$, $|A|=n-k-2$ and $\forall j\in A$ such that $\min A< j < \max A$ we have $j-1 \in A$ or $j+1 \in A$;
    \item \label{xiv} $|A|\leqslant n-k-3$ 
\end{enumerate}
\end{lemma}
\begin{proof}
We split the proof into the following cases: 
\begin{enumerate}
    \item First, we consider $|A|=n-k-1$. If $1\in A$ and $2\notin A$, then \ref{i} holds. If $1\in A$, $2\in A$ and $n-1\notin A$, then \ref{ii} holds. On the other hand, if $1, 2, n-1\in A$, then \ref{iii} holds when $n\notin A$ and \ref{iv} holds when $n\in A$. If $1, n\notin A$, then \ref{v} holds. If $1\notin A$ and $n\in A$, then \ref{ii} holds when $n-1\notin A$ and \ref{iv} holds when $n-1\in A$.
    \item Next, we consider $|A|=n-k-2$. If $A\cap\{n-1, n\}=\emptyset$ and $A\cap\{1, 2\}=\emptyset$ then $A$ satisfies either \ref{xii} or \ref{xiii} . On the other hand, if $A\cap\{n-1, n\}=\emptyset$ and $A\cap\{1, 2\}\neq\emptyset$ then \ref{vii} holds when $A\setminus\{\max A\}\notin\F$ and \ref{x} holds when $A\setminus\{\max A\}\in\F$. Next, if $A\cap\{n-1, n\}\neq\emptyset$ and $A\cap\{1, 2\}=\emptyset$ then \ref{viii} holds when $A\setminus\{\min A\}\notin\F$ and \ref{xi} holds when $A\setminus\{\min A\}\in\F$. And if $A\cap\{n-1, n\}\neq\emptyset$ and $A\cap\{1, 2\}\neq\emptyset$, then \ref{ix} holds.
    \item Finally, when $|A|\leqslant n-k-3$, \ref{xiv} holds. \qedhere
\end{enumerate}
\end{proof}
Equipped with Lemma~\ref{sec:4-lem_1}, we now state Theorem~\ref{sec:4_theorem_1} which provides a description of the closures of $\F$.
\begin{theorem}\label{sec:4_theorem_1}
Given any positive integer $k\leqslant n-5$ and any subset $A$ of $[n]$ with $A\notin\F$ and $|A|\leqslant n-k-1$, if $A$ satisfies any of the conditions stated in \ref{i}, \ref{ii}, \ref{vii}, \ref{viii}, \ref{ix}, \ref{xii}, \ref{xiv} of Lemma~\ref{sec:4-lem_1}, then $A \notin \overline{\F}^{(k)}$. On the other hand, if $A$ satisfies any of the conditions \ref{iii}, \ref{iv}, \ref{v}, \ref{vi}, \ref{x}, \ref{xi}, \ref{xiii} then $A\in\overline{\F}^{(k)}$.   

\end{theorem}
\begin{proof} Define $N=\{ \ref{i}, \ref{ii}, \ref{vii}, \ref{viii}, \ref{ix}, \ref{xii}, \ref{xiv}\}$ and $M=\{\ref{iii}, \ref{iv}, \ref{v}, \ref{vi}, \ref{x}, \ref{xi}, \ref{xiii}\}$. For $q\in N$, define $P(l, q)$ be the statement that: if $A$ satisfies statement $q$ of Lemma~\ref{sec:4-lem_1} for $k=l$ then $A\notin\overline{\F}^{(l)}$. And for $q\in M$, define $P(l, q)$ be the statement that: if $A$ satisfies statement $q$ of Lemma~\ref{sec:4-lem_1} for $k=l$ then $A\in\overline{\F}^{(l)}$. Let 
$$P(k)=\bigwedge_{q\in N\cup M}P(k, q).$$
To prove the theorem, we have to prove $P(k)$ for $1\leq k\leq n-5$. The proof is via induction on $k$. Let us first establish the base case, i.e.\ $P(1)$. 
\begin{enumerate}
    \item 
When $A$ satisfies \ref{i} of Lemma~\ref{sec:4-lem_1}, we have $A\cap \{3, 4, \dots, n\}=\{3, \dots, n\} \setminus \{j\}$. If $j=3$ then $A\cup\{3, 4\}\notin\F$ and if $j>3$ then $A\cup\{j-1, j\}\notin\F$, and either scenario leads to $A\notin\overline{\F}$.
    \item Follows via an argument similar to the argument for $P(1, \ref{i})$, via symmetry.
    \item 
When $A$ satisfies \ref{iii} of Lemma~\ref{sec:4-lem_1}, there exists an $i$ such that $A=[n-1] \setminus \{i\}$. If $i<n-1$ then the only supersets of $A$ are $[n-1]$, $[n]$ and $A\cup\{n\}$, which are all in $\F$. If $i=n-1$ then $A=[n-2]$. Therefore, it suffices to consider $A\cup B$ for the member-subsets $B=\{n-2, n-1\}$ and $B=\{n-1, n\}$ of $\F$. In the former situation, we obtain $A\cup B=[n-1]$ and in the latter, $A\cup B=[n]$, thus proving the statement.
    \item Follows via an argument similar to the argument for $P(1, \ref{iii})$, via symmetry.
    \item 
When $A$ satisfies \ref{v} of Lemma~\ref{sec:4-lem_1}, we have $A=\{2, 3,\cdots, n-1\}\in\F\subseteq\overline{\F}$.
    \item 
When $A$ satisfies \ref{vi} of Lemma~\ref{sec:4-lem_1}, it is straightforward to check that all supersets of $A$ are already in $\mathcal F$, thus yielding $A\in\mathcal F^{(1)}$.
    \item 
    When $A$ satisfies \ref{vii} of Lemma~\ref{sec:4-lem_1} then the only possibility for $A$ is the set $A=\{1, 3, \dots , n-2\}$. And $A\cup\{n-1, n\}\notin\F\Rightarrow A\notin\overline{\F}$.
    \item Follows via an argument similar to the argument for $P(1, \ref{vii})$, via symmetry.
    \item 
    Suppose $A$ satisfies \ref{ix} of Lemma~\ref{sec:4-lem_1}. Since, $A\notin\F$, there is an $j\in A$ such that $j-1, j+1\notin A$. If there is an $i\leqslant j-3$ such that $i\in A$, $i+1\notin A$ then $A\cup\{i, i+1\}\notin\F$ therefore, $A\notin\overline{\F}^{(1)}$. Also, if $1\notin A$ and $1\leqslant j-3$, then the same argument works by considering $A\cup\{1,2\}$. Otherwise, suppose that for each $i\leqslant j-2$, $i\in A$. Since, $|A|=n-3$ and $A\cap\{n-1,n\}\neq\emptyset$, we must have $j\leqslant n-3$ and hence there is an $i\geqslant j+3$ such that $i\in A$ and ($i-1\notin A$ or $i+1\notin A$) which means that $A\cup\{i-1,i\}\notin\F$ or $A\cup\{i, i+1\}\notin\F$ respectively. In either case, $A\notin\overline{\F}$.
    \item 
    When $A$ satisfies \ref{x} of Lemma~\ref{sec:4-lem_1} then, the only such $A$ is $A=\{1, 2, \dots n-4, n-2\}$. Note that $A\cup B$ for $B\in\F$ can either be $A, [n], [n-1], [n-2], \{1, 2, \dots n-4, n-2, n-1\}$ or $\{1, 2, \dots n-4, n-2, n-1, n\}$. All of these possible sets are in $\F$ and as a result, we have that $A\in\overline{\F}$ ($A\cup B$ cannot be $\{1, 2, \dots n-4, n-2, n\}$ since $n\in A\cup B$ means $n\in B\in\F$ which means that $n-1\in B\subseteq A\cup B$).
    \item Follows via an argument similar to the argument for $P(1, \ref{x})$, via symmetry.
    \item For $k=1$, there is no $A$ satisfying \ref{xii}, hence there is nothing to check in this case.
    \item For $k=1$, there is no $A$ satisfying \ref{xiii}, hence there is nothing to check in this case. 
    \item We have $A\notin\F$ and $|A|\leqslant n-4$.
    Suppose $|A|=1$. If $A=\{i\}$ and $i\geq 4$ then consider $A\cup\{1, 2\}\notin\F$ and hence $A\notin\overline{\F}$. Otherwise, if $i\leqslant 3$, then considering $A\cup\{5, 6\}$ gives the required result, i.e.\ $A\notin\overline{\F}$. So, we can assume that $|A|\geq 2$. Now, since $A\notin\F$, we have a $j\in A$ such that $j-1, j+1\notin A$. Since $|A|\geq 2$, we may assume, without loss of generality, that there is an $i\in A$ such that $i<j$. Now, if there is an $a\leqslant j-3$ such that $a\in A$ but $a+1\notin A$ or $a-1\notin A$, then we have $A\cup\{a, a+1\}\notin\F$ or $A\cup\{a-1, a\}\notin\F$ respectively, which gives $A\notin\overline{\F}$. So, we can assume that $a\leqslant j-2\Rightarrow a\in A$. Now, if $j=\max A$ then because of the last assumption and the size of $A$, $j\leqslant n-3$. Therefore, $A\cup\{n-1, n\}\notin\F$. On the other hand if there is a $b>j$, then since $|A|=n-4$, there is a $c>j+2$ such that either $c-1\notin A$ or $c+1\notin A$, which in turn implies that either $A\cup\{c-1, c\}$ or $A\cup\{c, c+1\}$ is absent in $\F$, and hence $A\notin \overline{\F}$, as required.
\end{enumerate}
Now, suppose that $P(t)$ is true for some $t\leqslant n-6$. Let us consider the $(t+1)$-st case. By Corollary~\ref{sec:2:cor_1} we have that $\mathcal A_{n-t,n}\subseteq\overline{\F}^{(t)}$. 
\begin{enumerate}
    \item 
    Suppose $A$ satisfies \ref{i} of Lemma~\ref{sec:4-lem_1}. First, consider $A=\{1\}$. We have $\{3,4\}\in\F\subseteq\overline{\F}^{(t)}$, thus yielding $A\cup\{3, 4\}=\{1, 3, 4\}$. Moreover, we have $\{1, 3, 4\}\subseteq\{1, 3, \cdots, n\}$ and $1\in\{1, 3, 4\}$. Since $t\leqslant n-4,$ we get $|\{1, 3, 4\}|=3\leqslant n-t-1$. Therefore, by $P(t,\ref{i})$, we conclude that $A\cup\{3, 4\}\notin\overline{\F}^{(t)}$, which gives $A=\{1\}\notin\overline{\F}^{(t+1)}$, as required. Now, let us consider the case where $A\neq\{1\}$. Then, there must exist an $i\in\{3, 4,\cdots , n\}$ such that $i\in A$. Since $|A|\leqslant n-t-2\leqslant n-3$, there must exist a $j\in\{3, 4,\cdots, n\}$ such that $j\notin A$. Therefore, there exist consecutive $i, j\in\{3, 4, \cdots, n\}$ such that $i\in A$ and $j\notin A$. Thus, by $P(t,\ref{i})$, we have $\{i, j\}\cup A\notin\overline{\F}^{(t)}$ and hence $A\notin\overline{\F}^{(t+1)}$, as required.  
    \item Follows via an argument similar to the argument for $P(t+1, \ref{i})$, via symmetry.
    \item 
    Suppose $A$ satisfies \ref{iii} of Lemma~\ref{sec:4-lem_1}. We need to show that for each $B\in\overline{\F}^{(t)}$, we have $A\cup B\in\overline{\F}^{(t)}\cup\{A\}$. Since $|A|=n-t-2$ and $\mathcal A_{n-t,n}\subseteq\mathcal A$, we need only check the case where $|A\cup B|=n-t-1$. We must have $|A\cup B\setminus A|=|B\setminus A|=1$. Let $B\setminus A=\{i\}$. If $i\neq n$ then by $P(t,\ref{iii})$ of the induction hypothesis, we have $A\cup B\in\overline{\F}^{(t)}$. This yields $i=n$. Suppose $n-1\notin A$. Since $B\setminus A=\{n\}$, this means that $n-1\notin B$. Moreover, $|B|\leqslant |A\cup B|=n-t-1$. Therefore, by $P(t,\ref{ii})$, we conclude that $B\notin\overline{\F}^{(t)}$, which leads to a contradiction since we started with $B\in\overline{\F}^{(t)}$. Therefore, we must have, $\{1, 2, n-1, n\}\subset A\cup B$, and $|A\cup B|=n-t-1$, which, by $P(t,\ref{vi})$, gives us $A\cup B\in\overline{\F}^{(t)}$, as required. Therefore, $A\in\overline{\F}^{(t+1)}$.
    \item Follows via an argument similar to the argument for $P(t+1, \ref{iii})$, via symmetry.
    \item 
    Suppose $A$ satisfies \ref{v} of Lemma~\ref{sec:4-lem_1}. We need to check that for each $B\in\overline{\F}^{(t)}$, we have $A\cup B\in\overline{\F}^{(t)}\cup\{A\}$. As noted above in the argument for $P(t+1,\ref{iii})$, it is enough to check this for the case when $|A\cup B|=n-t-1$. We have $B\setminus A=\{i\}$ for some $i$. Suppose $i\neq 1, n$. Then $A\cup B\subseteq \{2, 3, \cdots, n-1\}$ and $|A\cup B|=n-t-1$ together yield, by $P(t,5)$, that $A\cup B\in\overline{\F}^{(t)}$. Without loss of generality, suppose $i=1$. If $2\notin A$, then since $B\setminus A=\{1\}$, we have $2\notin B$. But, in that case, we have $1 \in B\subseteq \{1, 3, 4, \cdots, n\}$ and $|B|\leqslant n-t-1$, which together imply, by $P(t,\ref{i})$, that $B\notin\overline{\F}^{(t+1)}$, which is a contradiction to our hypothesis. Therefore, we must have $2\in A$. Therefore, $\{1, 2\} \subset A\cup B\subseteq\{1, 2, \cdots , n-1\}$ and $|A\cup B|=n-t-1$, so that by $P(t,\ref{iii})$, we have $A\cup B\in\overline{\F}^{(t)}$, which in turn gives us $A\in\overline{\F}^{(t+1)}$. 
    \item 
    Suppose $A$ satisfies \ref{vi} of Lemma~\ref{sec:4-lem_1}. We need to check that for each $B\in\overline{\F}^{(t)}$, we have $A\cup B\in\overline{\F}^{(t)}\cup\{A\}$. As noted above in the argument for $P(t+1,\ref{iii})$, it is enough to check this for the case when $|A\cup B|=n-t-1$. In this case, $\{1, 2, n-1, n\}\subseteq A\cup B$ and $|A\cup B|=n-t-1$ together imply, by $P(t,\ref{vi})$, that $A\cup B\in\overline{\F}^{(t)}$. Therefore, $A\in\overline{\F}^{(t+1)}$, as required.
    \item 
    Suppose $A$ satisfies \ref{vii} of Lemma~\ref{sec:4-lem_1}. Since $A-\{\max A\}\notin\F$, there exists a $j<\max A$ such that $j-1, j+1\notin A$. Therefore, $A\cup\{\max A, \max A+1\}\notin\F$. Moreover, $A\neq A\cup\{\max A, \max A+1\}$. If $\max A=n-2$, then $A\cup\{\max A, \max A+1\}\notin\overline{\F}^{(t)}$, by $P(t,\ref{ix})$. Otherwise, if $\max A<n-2$, then $A\cup\{\max A, \max A+1\}\notin\overline{\F}^{(t)}$, by $P(t,\ref{vii})$. Either way, we have $A\notin \overline{\overline{\F}^{(t)}}=\overline{\F}^{(t+1)}$.
    \item Follows via an argument similar to the argument for $P(t+1, \ref{vii})$, via symmetry.
    \item 
    Suppose $A$ satisfies \ref{ix} of Lemma~\ref{sec:4-lem_1}. Since $A\notin\F$, there exists a $j\in A$ such that $j-1, j+1\notin A$. If there is an $i\leqslant j-3$ such that $i\in A$ and $i+1\notin A$, then by $P(t,\ref{ix})$, we have $A\cup\{i, i+1\}\notin\overline{\F}^{(t)}\Rightarrow A\notin\overline{\F}^{(t+1)}$. Moreover, if $1\notin A$ and $1\leqslant j-3$, then by considering $A\cup\{1,2\}$, the same argument works. Otherwise, suppose $i\in A$ for each $i\leqslant j-2$. Since $|A|=n-t-3\leqslant n-4$ and $A\cap\{n,n+1\}\neq\emptyset$, we must have $j\leqslant n-3$ and hence there is an $i\geq j+3$ such that $i\in A$ and ($i-1\notin A$ or $i+1\notin A$) which means that $A\cup\{i-1,i\}\notin\F$ or $A\cup\{i, i+1\}\notin\F$ respectively. In either case, $A\notin\overline{\F}^{(t)}$ by $P(t,\ref{ix})$.
    \item Suppose $A$ satisfies \ref{x} of Lemma~\ref{sec:4-lem_1}. We need to check that for each $B\in\overline{\F}^{(t)}\cup\{A\}$, we have $A\cup B\in\overline{\F}^{(t)}$. Since, $|A|=n-t-3$, we only need to check the cases $|A\cup B|=n-t-2$ and $|A\cup B|=n-t-1$. Since $B\in\overline{\F}^{(t)}$, we divide the proof into the following parts:
    \begin{enumerate}
        \item First, we consider the case where $B\in\F$. Note that if $\max B>\max A+1$, then $\max B-1\notin A$ and hence we can assume that $|A\cup B|=n-t-1$, which in turn implies that $A\cup B=A\cup\{\max B-1, \max B\}$. If $\max B\leqslant n-1$ and $1\notin A$, then $A\cup B\in\overline{\F}^{(t)}$ by $P(t,\ref{v})$, and if $\max B=n$ and $1\notin A$, then $A\cup B\in\overline{\F}^{(t)}$ by $P(t,\ref{iv})$. On the other hand, if $1\in A$, then $2\in A$ because $A\setminus\{\max A\}\in\F$. So, if $1\in A$ and $\max B\leq n-1$, then $A\cup B\in\overline{\F}^{(t)}$ by $P(t,\ref{iii})$, and if $1\in A$ and $\max B=n$, then $A\cup B\in\overline{\F}^{(t)}$ by $P(t,\ref{vi})$. If $\max B\in\{\max A-1, \max A, \max A+1\}$, then from $A\setminus\{\max A\}\in\F$, it follows that $A\cup B\in\F\subseteq\overline{\F}^{(t)}$. Finally, we consider the case where $\max B\leq \max A-1$. We first assume that $|A\cup B|=n-t-1$. If $1\notin A$ and $1\notin B$, then $A\cup B\in\overline{\F}^{(t)}$ by $P(t,\ref{v})$. Now, recall that if $1\in A$, then $2\in A$ and if $1\in B$ then $2\in B$. In both these cases, $A\cup B\in\overline{\F}^{(t)}$ by $P(t,\ref{iii})$. On the other hand, if $|A\cup B|=n-t-2$, then it is straightforward to see that $A\cup B\setminus\{\max(A\cup B\}=A\cup B\setminus\{\max A\}\in\F$. Therefore, by $P(t, \ref{x})$, we have $A\cup B\in\overline{\F}^{(t)}$.
        \item Next, we consider the case where $B$ satisfies \ref{x} of Lemma~\ref{sec:4-lem_1}. Note that since $|B|=n-t-2$, we can assume that $|A\cup B|=n-t-1$. So, if $1\notin A$ and $1\notin B$, then $A\cup B\in\overline{\F}^{(t)}$ by $P(t,\ref{v})$. And if $1\in A$ or $1\in B$, then $2\in A\cup B$, and therefore, $A\cup B\in\overline{\F}^{(t)}$ by $P(t,\ref{iii})$.
        \item  We consider the case where $B$ satisfies \ref{xi} of Lemma~\ref{sec:4-lem_1}. Since, $|B|=n-t-2$, we can assume that $|A\cup B|=n-t-1$. Suppose, if possible, $A\cup B\notin\overline{\F}^{(t)}$. By Lemma~\ref{sec:4-lem_1} and the induction hypothesis, $A\cup B$ either satisfies \ref{i} or \ref{ii}. If it satisfies \ref{i}, then we must have $1\in A\cup B$ and $2\notin A\cup B$. But $1\in A\cup B$ implies that $1\in A$ (since $B\subseteq\{3, 4, \dots, n\}$), and $1\in A$, in turn, implies that $2\in A\subseteq A\cup B$. A similar argument works for the case where $A\cup B$ satisfies \ref{ii}.
        \item Finally, we consider the case where $B$ satisfies \ref{xiii} of Lemma~\ref{sec:4-lem_1}. We may once again assume that $|A\cup B|=n-t-1$. Note that since $B\subseteq\{3, 4, \dots, n-2\}$, if $1\notin A$ then $A\cup B\in\overline{\F}^{(t)}$ by $P(t, \ref{v})$, and if $1\in A$ then $A\cup B\in\overline{\F}^{(t)}$ by $P(t, \ref{iii})$.
    \end{enumerate}
    \item Follows via an argument similar to the argument for $P(t+1, \ref{x})$, via symmetry.
    \item If $A$ satisfies \ref{xii} of Lemma~\ref{sec:4-lem_1} and $\min A=3$, then $\{\min A-1, \min A\}\cup A\notin \overline{\F}^{(t)}$ by $P(t,\ref{vii})$. On the other hand, if $\min A>3$, then $\{\min A-1, \min A\}\cup A\notin\overline{\F}^{(t)}$ by $P(t, \ref{xii})$.  
    \item Suppose $A$ satisfies \ref{xiii} of Lemma~\ref{sec:4-lem_1}. We need to check that for each $B\in\overline{\F}^{(t)}$, we have $A\cup B\in\overline{\F}^{(t)}\cup\{A\}$. Since, $|A|=n-t-3$, we only need to check the cases where $|A\cup B|=n-t-2$ and $|A\cup B|=n-t-1$. Since $B\in\overline{\F}^{(t)}$, we divide the proof into the following parts:
    \begin{enumerate}
        \item First, we consider the case where $B\in\F$. Let us first suppose that $\min A-1< \min B<\max B<\max A$, which gives $\max A\cup B=\max A$ and $\min A\cup B=\min A$. Suppose, first that $|A\cup B|=n-t-2$. Consider $j\in A\cup B$ such that $\min A< j< \max A$. If $j\in A$ then $j-1\in A\cup B$ or $j+1\in A\cup B$ because of \ref{xiii}. On the other hand, if $j\in B$ then because $B\in\F$, we obtain that $j-1\in B$ or $j+1\in B$. Therefore, $A\cup B$ satisfies $\ref{xiii}$ and hence $A\cup\in\overline{\F}^{(t)}$ by $P(t, \ref{xiii})$. Next, if $|A\cup B|=n-t-1$ then $A\cup B\in\overline{\F}^{(t)}$ by $P(t, \ref{v})$. On the other hand, let us consider the case $\min A-1<\min B<\max B=\max A+1$ and $A\cup B\notin\F$. If $|A\cup B|=n-t-2$ then, it is easy to see that $A\cup B\setminus\{\min (A\cup B)\}=A\cup B\setminus\{\min A\}\in\F$ and hence $A\cup B\in\overline{\F}^{(t)}$ by $P(t,\ref{xi})$. And if $|A\cup B|=n-t-1$ then, $A\cup B\in\overline{\F}^{(t)}$ by $P(t,\ref{v})$. Analogous arguments go through for the case when $\min A-1=\min B<\max B<\max A+1$. Now, if $\min A-1=\min B<\max B=\max A+1$ then since $A$ satisfies \ref{xiii}, it is easy to see that $A\cup B\in\overline{\F}^{(t)}$. Finally, suppose that $\max B>\max A+1$. In this case, $\max B-1, \max B\notin A$ and since $B\in\F$, we do have that $\max B-1\in B$. Also, since we are assuming that $|A\cup B|\leqslant n-t-1$, we have that $A\cup B=A\cup\{\max B-1, \max B\}$. If $\max B\leqslant n-1$ then, $A\cup B\in\overline{\F}^{(t)}$ by $P(t,\ref{v})$ and if $\max B=n$ then $A\cup B\in\overline{\F}^{(t)}$ by $P(t,\ref{iv})$. Analogous argument works for the case when $\min B<\min A-1$. 
        \item Next, we consider the case where $B$ satisfies \ref{x} of Lemma~\ref{sec:4-lem_1}. It means that $|B|=n-t-2$ therefore, we can assume that $|A\cup B|=n-t-1$. If $\min B\geqslant 2$ then $A\cup B\in\overline{\F}^{(t)}$ by $P(t,\ref{v})$ and on the other hand, if $\min B=1$ then, since $B\setminus\{\max B\}\in\F$ and $t\leqslant n-5\Rightarrow |B|=n-t-2\geqslant 3$, we can conclude that $2\in B$. But note that $1, 2\notin A$ which give us that $|A\cup B|\geqslant n-k$. 
        \item The case when $B$ satisfies \ref{xi} of Lemma~\ref{sec:4-lem_1} is analogous to the previous part, via symmetry.
        \item Finally, we consider the case when $B$ satisfies \ref{xiii} of Lemma~\ref{sec:4-lem_1}. Since, in this case, we have that $|B|=n-t-2$, we can again assume that $|A\cup B|=n-t-1$ which, by $P(t,\ref{v})$, gives us that $A\cup B\in\overline{\F}^{(t)}$. 
    \end{enumerate}
    \item Here, we have $A\notin\F$ and $|A|\leqslant n-t-4$. Suppose $|A|=1$. Therefore, if $A=\{i\}$ and $i\geq 4$, then consider $A\cup\{1, 2\}\notin\F$, which gives us $A\cup\{1, 2\}\notin \overline{\F}^{(t)}$ by $P(t,\ref{xiv})$. Otherwise, if $i\leqslant 3$, then considering $A\cup\{5, 6\}$ gives us the required result, i.e.\ $A\notin\overline{\F}^{(t+1)}$. So now, we may assume that $|A|\geq 2$. Since $A\notin\F$, there exists a $j\in A$ such that $j-1, j+1\notin A$. Since $|A|\geq 2$, we can assume, without loss of generality, that there is an $i\in A$ such that $i<j$. Now, if there is an $a\leqslant j-3$ such that $a\in A$ but $a+1\notin A$ or $a-1\notin A$, then $A\cup\{a, a+1\}$ or $A\cup\{a-1, a\}$, respectively, satisfies \ref{xiv}. These are, therefore, not in $\overline{\F}^{(t)}$ by $P(t,\ref{xiv})$, which gives $A\notin\overline{\F}^{(t+1)}$. So now, we may assume that $a\leqslant j-2\implies a\in A$. If $j=\max A$, then because of the last assumption and the size of $A$, we have $j\leqslant n-t-3\leqslant n-4$. If $|A|=n-t-4$, then $|A\cup\{n-1, n\}|=n-t-2$ and hence $A\cup\{n-1, n\}\notin\overline{\F}^{(t)}$ by $P(t,\ref{ix})$. If $|A|\leqslant n-t-5$, then $|A\cup\{n-1, n\}|\leqslant n-t-3$ and hence $A\cup\{n-1, n\}\notin\overline{\F}^{(t)}$ by $P(t,\ref{xiv})$. On the other hand, if there is a $b>j$ with $b\in A$, then since $|A|=n-t-4\leqslant n-5$, there is a $c>j+2$ such that $c-1\notin A$ or $c+1\notin A$, which means that $A\cup\{c-1, c\}$ or $A\cup\{c, c+1\}$, respectively, satisfies \ref{xiv}, and hence $A\notin\overline{\F}^{(t+1)}$, as required.
        \end{enumerate}
This brings us to the end of the proof that $P(k)$ is true for all $k\leqslant n-5$.
\end{proof}

Let us note that Lemma~\ref{sec:4-lem_1}, together with Theorem~\ref{sec:4_theorem_1}, describes all $\overline{\F}^{(k)}$'s for $k\leqslant n-5$. This is because we know that $\F\subseteq\overline{\F}^{(k)}$ and $\mathcal A_{n-k}\subseteq\overline{\F}^{(k)}$. Therefore, if we have an $A\subseteq 2^{[n]}$ such that $A\notin\F$ and $|A|\leqslant n-k-1$, then by Lemma~\ref{sec:4-lem_1}, it satisfies one of the statements \ref{i}, \ref{ii}, \ref{iii}, \ref{iv}, \ref{v}, \ref{vi}, \ref{vii}, \ref{viii}, \ref{ix}, \ref{x}, \ref{xi}, \ref{xii}, \ref{xiii},\ref{xiv}, but then that statement together with the corresponding statement in Theorem~\ref{sec:4_theorem_1} tells us whether $A\in\overline{\F}^{(k)}$ or $A\notin\overline{\F}^{(k)}$. 

We now come to the proof of the fact that $\F$ is $(n-1)$-dense.
\begin{cor}
$\F$ is $(n-1)$-dense.
\end{cor}
\begin{proof}
First, we establish that $\{3\}\in\overline{\F}^{(n-4)}$. Since $\mathcal A_{5, n}\subseteq\overline{\F}^{(n-5)}$, we only need to investigate $\{3\}\cup B$ for all $B\in\overline{\F}^{(n-5)}$ such that $|\{3\}\cup B|\leqslant 4$. The cases where $|\{3\}\cup B|=1, 2$ are trivial. Let us consider the cases where $|\{3\}\cup B|=3$. If $|B|=3$, then $\{3\}\cup B=B\in\overline{\F}^{(n-5)}$. Therefore, let $|B|=2$. By $P(n-5,\ref{xiv})$, this means that $B\in\F$. Therefore, let $B=\{i, i+1\}$. If $i\leqslant 4$, then $\{3\}\cup \{i, i+1\}\in\F\subseteq\overline{\F}^{(n-5)}$. If $i>4$, then by $P(n-5,\ref{xi})$, we have $\{3\}\cup B\in\overline{\F}^{(n-5)}$. The only case that is left to consider is $|\{3\}\cup B|=4$. Again, we may assume that $|B|=3$ and $3\notin B$. Therefore, $B$ satisfies either \ref{x} or \ref{xi}. If $B$ satisfies \ref{x}, then $B\notin\F$ and $B\setminus\{\max B\}\in\F$. Since $|B|=3$, this means that $B=\{i, i+1, j\}$ such that $i+1<j\leqslant n-2$. If $i=1$, then $\{3\}\cup B\in\overline{\F}^{(n-5)}$ by $P(n-5,\ref{iii})$; if $i\geq 4$ then $\{3\}\cup B\in\overline{\F}^{(n-5)}$ by $P(n-5,\ref{v})$; if $i=2$ or $i = 3$, then $\{3\}\cup B=B$. On the other hand, if $B$ satisfies \ref{xi}, then $B$ is of the form $\{j, i-1, i\}$ where $3\leqslant j<i-1$. If $i=n$, then $\{3\}\cup B\in\overline{\F}^{(k)}$ because of $P(n-5,\ref{iv})$, and if $i<n$ then $\{3\}\cup B\in F^{(k)}$ because of $P(n-5,\ref{v})$. Therefore, $\{3\}\in\overline{\F}^{(n-4)}$. \\
Since $\{1, 3, 4\}\cup\{4, 5\}=\{1, 3, 4, 5\}\notin\F^{(n-5)}$, by $P(n-5,\ref{i})$, we have $\{1, 3, 4\}\notin\overline{\F}^{(n-4)}$.
Now, consider the subsets, $\{3\}\subseteq\{3, 4\}\subseteq\{1, 3, 4\}$. Here, $\{3\}, \{3, 4\}\in\overline{\F}^{(n-4)}$ and $\{1, 3, 4\}\notin\overline{\F}^{(n-4)}$. Therefore, by Theorem~\ref{main_1}, $\overline{\F}^{(n-4)}$ is at least $3-$dense, which means that $\F$ is at least $(n-1)$-dense. This concludes the proof.
\end{proof}

\section{Relative subsets and closure roots}\label{sec:3}
\begin{defn}\label{relative_subset_defn}
Given a union-closed family $\mathcal{F}$ over the universe $[n]$, and $A, B\in\mathcal F$, we define $A$ to be a \emph{subset of $B$ relative to $\mathcal F$}, and write $A\subseteq_{\mathcal F}B$, if at least one of the following happens:
\begin{itemize}
    \item $A=B$.
    \item $B=[n]$
    \item $\exists C\in\mathcal F$ such that $C\neq B$ and $A\cup C=B$.
\end{itemize}
\end{defn}
We write $A\subsetneq_{\mathcal F}B$ if $A\subseteq_{\mathcal F}B$ and $A\neq B$, and $A\nsubseteq_{\mathcal F}B$ when $A$ is not a subset of $B$ relative to $\mathcal{F}$. 

Note that $A\subseteq_{\mathcal F}B$ implies that $A\subseteq B$. It is straightforward to see that this definition coincides with the usual notion of subsets if $\mathcal F=2^{[n]}$. To see that it actually differs from the usual notion of subsets when $\mathcal{F}$ is not $2^{[n]}$, it is enough to consider$\mathcal F=\{\{1\},\{1,2\},\{2,3\},\{1,2,3\}\}$ and observe that $\{1\}\nsubseteq_{\mathcal F}\{1,2\}$. 

\begin{lemma}\label{sec:3-lem_1}
Let $\mathcal F$ be a $1$-dense family. Consider any $A, B, C\in\mathcal F$. Then the following hold:
\begin{enumerate}
\item \label{part_1} $A\subsetneq_{\mathcal F}B$ and $B\subseteq C$ together imply that $A\subsetneq_{\mathcal F}C$.
\item \label{part_2} $A\subsetneq B$ and $B \subseteq C$ together imply that $B\subseteq_{\mathcal F}C$.
\end{enumerate} 
\end{lemma}
\begin{proof}
We prove the two parts separately, as follows.
\begin{enumerate}
\item If $C=B$, the conclusion is immediate. Assume, therefore, that $B \subsetneq C$. By definition of the relation $A\subsetneq_{\mathcal F}B$ , there exists $B_1\in\mathcal F$ such that $A\cup B_1=B$ and $B_1\subsetneq B$. Note that on one hand, $(C \setminus B)\cup B_1\subsetneq C$, and on the other, $B_1\subseteq (C \setminus B)\cup B_1$. The latter implies that $(C \setminus B)\cup B_1\in\mathcal F$ (since by Lemma~\ref{sec:2-lem_4}, we know that $\mathcal F$ is an up-set.) Moreover, $A\cup(C \setminus B)\cup B_1=(A\cup B_1)\cup(C \setminus B)=B\cup(C \setminus B)=C$. Together, these observations allow us to conclude that $A\subsetneq_{\mathcal F}C$.
\item If $B=C$, then there is nothing to prove. Otherwise, note that the set $A\in\mathcal F$ and from Lemma~\ref{sec:2-lem_4}, we know that $\mathcal{F}$ is an up-set, and these two together imply that $A\cup(C \setminus B)\in\mathcal F$. Moreover, $A\neq B$ implies that $A\cup(C-B)\neq C$, and $C=B\cup(A\cup (C \setminus B))$. Combining all these, we conclude that $B\subseteq_{\mathcal F}C$. \qedhere
\end{enumerate}
\end{proof}

\begin{cor}\label{sec:3-cor_1}
The relation $\subseteq_{\mathcal{F}}$ of relative subset is transitive when the family $\mathcal{F}$ is $1$-dense.
\end{cor}
This follows from using the fact that $B\subseteq_{\mathcal F}C$ implies $B\subseteq C$, and using \ref{part_1} of Lemma~\ref{sec:3-lem_1}. However, this is not necessarily true when $\mathcal{F}$ is not $1$-dense. As an instance, if we consider $\mathcal F=\{\{1\}, \{2\}, \{1,2\}, \{1,3\}, \{1,2,3\}, \{1,2,3,4\}\}$, then it can be checked that $\{1\}\subseteq_{\mathcal F}\{1,2\}\subseteq_{\mathcal F}\{1,2,3\}$ but $\{1\}\not\subseteq_{\mathcal F}\{1,2,3\}$. It would be interesting to try to characterize union-closed families $\mathcal{F}$ such that $\subseteq_{\mathcal{F}}$ is transitive.

\begin{prop}\label{sec:3_prop_1}
Let $\mathcal F$ and $\mathcal H\subseteq\mathcal F$ be two union-closed families over the universe $[n]$. Then $\overline{\mathcal H}\supseteq\mathcal F$ if and only if for all $A\in\mathcal H$ and $B\in\mathcal F$ with $A\subseteq_{\mathcal F}B$, we have $B\in\mathcal H$. 
\end{prop}
\begin{proof}
Suppose $\overline{\mathcal H}\supseteq\mathcal F$. Consider any $A\in\mathcal H$ and $B\in\mathcal F$ with $A\subseteq_{\mathcal F}B$. If $A = B$ or $B = [n]$, the conclusion is immediate. Therefore, we consider the third possibility in Definition~\ref{relative_subset_defn}, i.e.\ there exists $C\in\mathcal F$ such that $A\cup C=B$ and $C \neq B$. Now, $C\in\mathcal F \subseteq \overline{\mathcal H}$, which implies that $\mathcal{H} \cup \{C\}$ is union-closed. Since $A\in\mathcal H$, we must have $B = A\cup C\in\mathcal H\cup\{C\}$, and $B \neq C$. Therefore, we must have $B\in\mathcal H$, as claimed.

We now prove the converse. Consider $C\in\mathcal F$. We want to show that $C\in\overline{\mathcal H}$, i.e.\ the family $\mathcal{H} \cup \{C\}$ is union-closed. Thus, for every $A \in \mathcal{H}$, we want to show that the subset $A \cup C$ either equals $C$ or is in $\mathcal{F}$. If $A\cup C=A$ or $A\cup C=C$ then there's nothing further to check. Otherwise, $A\cup C=B$ for some $B$ with $B\neq C$ and $C \in \mathcal{F}$, thus establishing that $A \subseteq_{\mathcal{F}} B$. Moreover, as $A \in \mathcal{H} \subseteq \mathcal{F}$ and $C \in \mathcal{F}$ and $\mathcal{F}$ is union-closed, hence $B \in \mathcal{F}$ as well. Therefore, the hypothesis of the proposition implies that $B \in \mathcal H$, hence completing the proof. 
\end{proof}

It is worth noting how Proposition~\ref{sec:3_prop_1} compares with Lemma~\ref{sec:2-lem_4}. In Proposition~\ref{sec:3_prop_1}, if we take $\mathcal F=2^{[n]}$, then $\mathcal H$ has to be a $1$-dense family and $\forall A\in\mathcal H$ and $B\in 2^{[n]}$ such that $A\subseteq B$, we have $B\in\mathcal H$, which implies that $\mathcal H$ is an up-set, thus corroborating the claim of Lemma~\ref{sec:2-lem_4}. Thus Proposition~\ref{sec:3_prop_1} is a generalisation of Lemma~\ref{sec:2-lem_4} when we consider the notion of relative subsets. Let us recall that given a family $\F$, $A\in\F$ is  an \textit{inclusion-wise minimal} member subset if for all $B\in\F$ with $B\subseteq A$, we have $A=B$.

\begin{defn}\label{generating_set_defn}
Consider the family $\mathcal G$ of all inclusion-wise minimal member subsets of $\mathcal F$. When $\mathcal{F}$ is $1$-dense, using Lemma~\ref{sec:2-lem_4} we see that 
$$\mathcal F=\{B\in 2^{[n]}: \exists A\in\mathcal G \text{ with } B\supseteq A\}.$$
We write $<\mathcal G>=\mathcal F$ and say that $\mathcal G$ is the \textit{generating set} of $\F$.
\end{defn}

Let us recall the notion of basis sets of a union-closed family(see \cite{c}). A set $A\in\F$ is called a $\textit{basis set}$ if for all $X, Y\in\F$ with $A=X\cup Y$, we have either $X=A$ or $Y=A$. Let $\mathcal B$ be the set of all basis sets of $\F$. Then $\mathcal B$ is the minimal set with the property that the union-closed family generated by $\mathcal B$ is $\F$.
We emphasize here that $\mathcal G\neq\mathcal B$ in general. For example, if the $1$-dense family is $\{\{1,\},\{1,2\},\{1,3\},\{1,2,3\}\}$ then $\mathcal G=\{\{1\}\}$ but $\mathcal B=\{\{1\},\{1,2\},\{1,3\}\}$.  

\begin{defn}\label{generating_set_defn_1}
Let $\mathcal{F}$ be a union-closed family over the universe $[n]$, and let $\mathcal K\subseteq\mathcal F$. Define the family
\begin{equation}
<\mathcal K>_{\mathcal F}=\{B\in\mathcal F: \text{ there exists } A\in\mathcal K \text{ with } A\subseteq_{\mathcal F}B\}.
\end{equation}
We say that $\mathcal K$ generates $<\mathcal K>_{\mathcal F}$ relative to $\mathcal F$.
\end{defn}
We emphasise here that $\mathcal K\subseteq<K>_{\F}$. We shall eventually establish that when $\mathcal{G}$, as mentioned above, is the family of all inclusion-wise minimal member subsets of the $1-$dense family $\mathcal F$, the family $<\mathcal G>_{\mathcal F}$ serves as a test candidate for determining whether there exists a union-closed family $\mathcal H$ such that $\overline{\mathcal H}=\mathcal F$. 

\begin{lemma}\label{sec:3-lem_2}
Let $\mathcal F$ be a $1$-dense union-closed family and $\mathcal K\subseteq\mathcal F$. Then, $<\mathcal K>_{\mathcal F}$ is union-closed.
\end{lemma}
\begin{proof}
Fix any $A, B\in<\mathcal K>_{\mathcal F}$. Then there exist subsets $A_o$ and $B_o$ in $\mathcal K$ such that $A_o\subseteq_{\mathcal F}A$ and $B_o\subseteq_{\mathcal F}B$. If $A_o=A$ and $B_o=B$ and $A_o\cup B_o$ equals either $A_{o}$ or $B_{o}$, then the conclusion is immediate since $A_{o}, B_{o} \in \mathcal{K} \subset <\mathcal K>_{\mathcal F}$. If $A_{o} = A$ and $B_{o} = B$ but $A_{o} \cup B_{o}$ equals neither $A_{o}$ nor $B_{o}$, then since $B_{o} \in \mathcal{K} \subset \mathcal{F}$ and $B_{o} \neq A_{o} \cup B_{o}$, hence we conclude that $A_{o} \subset_{\mathcal{F}} A_{o} \cup B_{o}$. Finally, assume that either $A_{o} \subsetneq_{\mathcal{F}} A$ or $B_{o} \subsetneq_{\mathcal{F}} B$ or both. In this case, since $\mathcal{F}$ is a $1$-dense union-closed family, $A_{o}$, $A$ and $A \cup B$ all belong to $\mathcal{F}$, and $A_{o} \subsetneq_{\mathcal{F}} A \subset A \cup B$, by \ref{part_1} of Lemma~\ref{sec:3-lem_1}, we conclude that $A_o \subseteq_{\mathcal F}A\cup B\Rightarrow A\cup B\in<\mathcal K>_{\mathcal F}$. This concludes the proof.
\end{proof}

Given union-closed families $\H\subseteq\F$ and $A\in\H$, let us say that $A$ is \textit{$\subseteq_{\F}-$wise minimal member-subset} of $\H$ if whenever $B\in\H$ and $B\subseteq_{\F}A$, we have that $B=A$.

\begin{prop}\label{sec:3_prop_2}
Let $\mathcal F$ be a $1$-dense family and $\mathcal H\subseteq\mathcal F$ be union-closed. Then, the following are equivalent:
\begin{enumerate}
    \item \label{statement_1} $\overline{\mathcal H}\supseteq\mathcal F$.
    \item \label{statement_2} For all $A\in\mathcal H$ and $B\in\mathcal F$ such that $A\subseteq_{\mathcal F}B$, we have $B\in\mathcal H$.
    \item \label{statement_3} The set $\mathcal K\subseteq\mathcal H$ of all $\subseteq_{\F}-$wise minimal member-subsets of $\H$ satisfies $<\mathcal K>_{\mathcal F}=\mathcal H$.
\end{enumerate}
\end{prop}
\begin{proof}
That \ref{statement_1} and \ref{statement_2} are equivalent follows from Proposition~\ref{sec:3_prop_1}.

We now show that \ref{statement_3} implies \ref{statement_2}. Consider any $A\in<\mathcal K>_{\mathcal F}=\mathcal H$. Then there exists $A_o\in\mathcal K$ such that $A_o\subseteq_{\mathcal F}A$. Suppose there exists $B\in\mathcal F$ be such that $A\subseteq_{\mathcal F}B$. Then the relations $A_o\subseteq_{\mathcal F}A$ and $A\subseteq_{\mathcal F}B$ along with Corollary~\ref{sec:3-cor_1} imply that $A_o\subseteq_{\mathcal F}B$, therefore yielding $B\in<\mathcal K>_{\mathcal F}$. 

Finally, we show that \ref{statement_2} implies \ref{statement_3}. Consider any $B\in\mathcal H$. By minimality, there exists $A\in\mathcal K$ such that $A\subseteq_{\mathcal F}B$. This implies that $B\in<\mathcal K>_{\mathcal F}$. Conversely, consider any $B\in<\mathcal K>_{\mathcal F}$. Then there exists $A\in\mathcal K$ such that $A\subseteq_{\mathcal F}B$, which then implies that $B\in\mathcal H$, by the hypothesis in \ref{statement_2}.
\end{proof}

\subsection{Closure Roots}\label{subsec:closure_roots}
\begin{defn}\label{closure_root_defn}
For a union-closed family $\mathcal F$, we say that a union-closed family $\mathcal H$ is a \emph{closure root} of $\mathcal F$ if $\overline{\mathcal H}=\mathcal F$.
\end{defn}
Closure roots need not exist nor do they need be unique: for example, $(n-1)$-dense families do not have closure roots, whereas every $1$-dense family is a closure root of $2^{[n]}$. In the tree $\mathcal G_n$ described in \S\ref{sec:2}, the families not having any closure root determine the leaves of the tree. In \S\ref{subsec:closure_roots}, we investigate the existence of closure roots of $1$-dense families. 
\begin{remark}Note that if $\mathcal H$ is a closure root of $\F$ then by Corollary~\ref{sec:2:cor_1}, we have that $\mathcal A_{n-1}\subseteq\overline{\mathcal H}=\F$. Therefore, any family $\F$ with $\mathcal A_{n-1,n}\nsubseteq\F$ does not have a closure root. In particular, this shows that every family considered in \S\ref{sec:k} does not have a closure root thus giving at least $\binom{n}{k-1}f_{k-1}$ many leaves at level $k$ of the tree $\mathcal G_n$. In Corollary~\ref{sec:3_example_lemma}, we give example of an $1-$dense family with $\mathcal A_{n-1}\subseteq\F$, which does not have a closure root.
\end{remark}
\begin{lemma}
Let $\mathcal F$ be a $1$-dense family, with $\mathcal G$ its generating set (as defined in Definition~\ref{generating_set_defn}). Let $A,B\in<\mathcal G>_{\mathcal F}$ with $A\subseteq B$. Then $A\subseteq_{\mathcal F}B$.
\end{lemma}
\begin{proof}
If $A=B$ or $B=[n]$, then the conclusion is immediate, therefore we assume that $A \subsetneq B \subsetneq [n]$. Note that as $B\in<\mathcal G>_{\mathcal F}$, there exists $C\in\mathcal G$ such that $C\subsetneq_{\mathcal F}B$, which in turn means that there exists $C_1\in\mathcal F$ such that $C\cup C_1=B$ and $C_1\subsetneq B$. Now, $A\subsetneq B$ implies that $A\subsetneq C\cup C_1$, and $C\in\mathcal G$ implies that $A\nsubseteq C$, which in turn yields $\emptyset\neq A \setminus C\subsetneq C_1 \setminus C$. Let $a\in A \setminus C$. We then have
$$A\cup(C\cup (C_1 \setminus \{a\}))=B$$ 
where $C\subseteq C\cup(C_1 \setminus \{a\})$ and $C \in \mathcal{G}$ together imply that $C\cup(C_1 \setminus \{a\})\in\mathcal F$, and since $a\notin C$, hence $C\cup(C_1-\{a\})\neq B$. These observations together yield $A\subseteq_{\mathcal F}B$, as desired.    
\end{proof}

\begin{theorem}\label{sec:3_theorem_1}
Let $\mathcal F$ be a $1$-dense family with generating set $\mathcal G$. Let $\mathcal H\subseteq\mathcal F$ be a union-closed family such that $\overline{\mathcal H}\supseteq\mathcal F$. Then $\overline{\mathcal H}\supseteq\overline{<\mathcal G>}_{\mathcal F}$.
\end{theorem}
\begin{proof}
Let $\mathcal K$ be the set of all $\subseteq_{\mathcal F}$-wise minimal elements of $\mathcal H$. Since $\overline{\mathcal H}\supseteq\mathcal F$, Proposition~\ref{sec:3_prop_2} yields $<\mathcal K>_{\mathcal F}=\mathcal H$. Let $A\in\overline{<\mathcal G>}_{\mathcal F}$ and $B\in<\mathcal K>_{\mathcal F}$. We wish to show that $A\cup B\in<\mathcal K>_{\mathcal F}\cup\{A\}$. If $A\cup B=B$ or $A\cup B=A$, then the conclusion is immediate. Henceforth, assume $A\nsubseteq B$ and $B\nsubseteq A$. Therefore, $B\subsetneq A\cup B$. There exists $C\in\mathcal K$ such that $C\subseteq_{\mathcal F}B$. If $C\neq B$, then $C\subsetneq_{\mathcal F}B\subseteq A\cup B$. Hence, by Lemma~\ref{sec:3-lem_1}, we get $C\subsetneq_{\mathcal F}A\cup B$, implying that $A\cup B\in<\mathcal K>_{\mathcal F}$. Now, suppose $C=B$. There exists $A_o\in\mathcal G$ such that $A_o\subseteq B$. Suppose $A_o\neq B$. We get $A_o\subsetneq B\subseteq A\cup B$, implying that $B\subseteq_{\mathcal F}A\cup B$, again by Lemma~\ref{sec:3-lem_1}. Therefore, we are left with $A_o=B$. Since $A\in\overline{<\mathcal G>}_{\mathcal F}$ and $A\cup B\neq A$, we get $A\cup B\in<\mathcal G>_{\mathcal F}$ and $B\subsetneq A\cup B$ with $B\in\mathcal G$. By Lemma~\ref{sec:3-lem_1}, $B\subseteq_{\mathcal F}A\cup B$. Hence, $A\cup B\in<\mathcal K>_{\mathcal F}\cup\{A\}$.    
\end{proof}

Equipped with Theorem~\ref{sec:3_theorem_1}, we characterise the $1$-dense families that have closure roots, via Theorem~\ref{sec:3_theorem_2}. 

\begin{theorem}\label{sec:3_theorem_2}
Let $\mathcal F$ be a $1$-dense family over universe $[n]$. Let $\mathcal G$ be the generating set of $\mathcal F$, as defined in Definition~\ref{generating_set_defn}. Then $\mathcal F$ has a closure root if and only if $\overline{<\mathcal G>}_{\mathcal F}=\mathcal F$.  
\end{theorem}
\begin{proof}
If $\overline{<\mathcal G>}_{\mathcal F}=\mathcal F$ then the conclusion is immediate. Suppose there exists a closure root $\mathcal H$ of $\mathcal F$. By Theorem~\ref{sec:3_theorem_1}, we have $\overline{\mathcal H}\supseteq\overline{<\mathcal G>}_{\mathcal F}\supseteq\mathcal F$. But $\overline{\mathcal H}=\mathcal F$ by Definition~\ref{closure_root_defn}, hence $\overline{<\mathcal G>}_{\mathcal F}=\mathcal F$.
\end{proof}
We mention here that the proofs of Theorems~\ref{sec:3_theorem_1} and \ref{sec:3_theorem_2} can not be generalised to general families because, as remarked earlier, transitivity of $\subseteq_{\mathcal F}$ does not necessarily hold when $\mathcal{F}$ is a $k$-dense family for some $k > 1$.  

We now provide the example of a $1-$dense family $\F$ with $\mathcal A_{n-1,n}\subseteq\F$ which does \emph{not} have a closure root. 
\begin{cor}\label{sec:3_example_lemma}
Let $k\geq 2$ and $\mathcal F$ be the $1$-dense family ,over universe $[2k+1]$, generated by $\mathcal G$ $= \{\{1, 2\}, \{2, 3\}, \dots, \{2k, 2k+1\}\}$. Then $\mathcal{F}$ does not have a closure root.
\end{cor}
\begin{proof}
 Consider $A=\{1, 3, 5,\dots, 2k+1\}$. It is immediate that $A \notin \mathcal F$. Let $B\supsetneq A$ and $a\in B \setminus A$. Since $1<a<2k+1$, we have $\{a-1, a, a+1\}\subseteq B$. Since $\{a-1, a\}\subsetneq_{\mathcal F}\{a-1, a, a+1\}\subseteq B$ therefore, by Lemma~\ref{sec:3-lem_1}, we have $\{a-1, a\}\subsetneq_{\mathcal F}B$, and as $\{a-1, a\} \in \mathcal{G}$, hence $B\in<\mathcal G>_{\mathcal{F}}$. Therefore, in particular $\forall C\in<\mathcal G>_{\F}$ with $A\cup C\supsetneq A$, we have $A\cup C\in<\mathcal G>_{\F}$ , which implies that $A\in\overline{<\mathcal G>}_{\mathcal{F}}$. Therefore, by Theorem~\ref{sec:3_theorem_2}, $\mathcal F$ does not have a closure root. 

\end{proof}
\begin{remark}
Let $\mathcal F$ be the $1$-dense family over universe $[2k]$ generated by $\mathcal G=\{\{1, 2\}, \{2, 3\}, \dots, \{2k-1, 2k\}\}$ for $k\geq 3$. Then $\mathcal{F}$ has a closure root.
\end{remark}
To see this, let us first note that whenever $A\in<\mathcal G>_{\F}$ with $|A|\geqslant 3$ then the number of sets $\{i, i+1\}$ such that $\{i, i+1\}\subseteq A$ is at least $2$. If $A=[2k]$ then the result is obvious, otherwise, by the definition of $<\mathcal G>_{\F},$ $A\in<\mathcal G>_{\F}$ means that there is a $\{i, i+1\}\in\mathcal G$ such that $\{i, i+1\}\subsetneq_{\F}A$. This means that there is a $B\in\F$ such that $\{i, i+1\}\cup B=A$ with $B\neq A$. This gives that $\{i, i+1\}\nsubseteq B$ and hence, since $B\in\F$ there is a $j\neq i$ such that $\{j, j+1\}\subseteq B\subseteq A$ hence proving our claim.\\
Now, let us consider $A\in\overline{<\mathcal G>_{\F}}\setminus\F$. Therefore, $\forall i\in A, i+1, i-1\notin A$. If $A=\{i\}$ with $i\geqslant 4$ then $\{1, 2\}\cup A\notin<\mathcal G>_{\F}\cup\{A\}$, otherwise, if $i\leqslant 3$ then $\{5, 6\}\cup A\notin<\mathcal G>_{\F}\cup\{A\}$. So, we can have that $|A|\geqslant 2$. Suppose $1\notin A$ and let $\text{min.} A=t$. Then by the discussion above, $\{t-1, t\}\cup A\notin<\mathcal G>_{\F}\cup\{A\}$. Therefore, $1\in A$. Now, suppose that for some $r\geqslant 2$, $\{1, 3,\dots, 2r-1\}\subseteq A$ and $2r, 2r+1\notin A$. But then, $\{2r-1, 2r\}\cup A\notin<\mathcal G>_{\F}\cup\{A\}$(by the discussion in the previous paragraph). Therefore, by induction, $\{1, 3, \dots, 2k-1\}\subseteq A$. But $\{2k-1, 2k\}\cup A\notin<\mathcal G>_{\F}$ and hence, $A\notin\overline{<\mathcal G>_{\F}}$ which gives us that $\overline{<\mathcal G>_{\F}}=\F$. Thus, completing the argument.\\

We would like to conclude by noting that the family $<\mathcal G>_{\F}$ as described above need not always be $2-$dense. In fact, it can be of arbitrary density. As an example if we consider $\F$ to be the $1-$dense family over universe $[n]$ generated by $\mathcal G=\{[k]\}$ where $n\geqslant k+2$. Then it is easy to see that the family $<\mathcal G>_{\F}=\{[k], [n]\}$ and hence by Corollary~\ref{mcor}, $<\mathcal G>_{\F}$ is $(k+1)-$dense. And finally if we let $\mathcal G=\mathcal A_{n-1, n}$ then $<\mathcal G>_{\F}=\mathcal F$ making $<\mathcal G>_{\F}$ to be $1-$dense.

\section{Further Questions}\label{sec:ques}

Here, we state some further questions, which we believe are worth exploring:
\begin{enumerate}
    \item Do $2$-dense families satisfy Conjecture 1?
    \item What is a characterization of $2$-dense families?
    \item Do $(n-1)$-dense families satisfy Conjecture 1?
    \item What is a characterization of $(n-1)$-dense families?
    \item Since it is known that union closed families $\F$ over universe $[n]$ satisfying $|\F|\geqslant 2^{n-1}$ satisfy Conjecture~\ref{main_conj}(see \cite{h}), it can be an interesting problem to characterise families satisfying $|\overline{\F}|\geqslant 2^{n-1}$.
    \item Suppose $\F$ satisfies Conjecture~\ref{main_conj}. Does $\overline{\F}$ also satisfy Conjecture~\ref{1}? (An affirmative answer to this problem would mean that if $\F$ is an counterexample to Conjecture~\ref{main_conj} then all its closure root are also counterexamples.)
    \item As in \cite{a}, let $d_{\mathcal F}(x)=|\{A\in\mathcal F|x\in A\}$. Let $g(\mathcal F)=\frac{1}{|\mathcal F|}\max\{d_{\mathcal F}(x)|x\in\mathcal F\}$. Let us put $a_{k,n}=\min\{g(\mathcal F)|\mathcal F$ is $k$-dense and over universe $[n]\}$. It can be an interesting problem to study the numbers $a_{k,n}$. If it turns out that $a_{k-1,n}\leqslant a_{k,n}$ for all possible values of $k$ and $n$ then it would prove Conjecture~\ref{main_conj} because $a_{0,n}\geqslant\frac{1}{2}$. We would like to note here that if Conjecture~\ref{main_conj} is true then $a_{n-1,n}\rightarrow\frac{1}{2}$ as $n\rightarrow\infty$. This follows by considering the family $\F=2^{[n-2]}\cup\{[n]\}$ and noting that it is $(n-1)-$dense using Corollary~\ref{mcor}. One natural way of trying to prove $a_{k-1,n}\leqslant a_{k,n}$ is to try and prove that for every family $\F$, $g(\F)\geqslant g(\overline{\F})$. But this is not true in general as is seen by again considering the family $\F=2^{[n-2]}\cup\{[n]\}$ and applying Theorem~\ref{mtheorem_1}.
    
\end{enumerate}
\section{Acknowledgements}
I would like to humbly thank my doctoral advisor Prof. Moumanti Podder for sharing her extremely helpful thoughts and for her help in improving the presentation of this paper.


\end{document}